\documentclass[10pt, reqno]{amsart}
\usepackage{amsmath, amsthm, amscd, amsfonts, amssymb, graphicx, color}
\usepackage[bookmarksnumbered, colorlinks, plainpages]{hyperref}
\hypersetup{colorlinks=true,linkcolor=red, anchorcolor=green, citecolor=cyan, urlcolor=red, filecolor=magenta, pdftoolbar=true}
\usepackage{mathrsfs}

\newtheorem{theorem}{Theorem}[section]
\newtheorem{lemma}[theorem]{Lemma}

\newtheorem{corollary}[theorem]{Corollary}
\theoremstyle{definition}
\newtheorem{definition}[theorem]{Definition}

\theoremstyle{remark}
\newtheorem{remark}[theorem]{Remark}
\numberwithin{equation}{section}

\begin{document}
\setcounter{page}{1}

\title[Local factors and Cuntz-Pimsner algebras]{Local factors and Cuntz-Pimsner algebras}

\author[Igor V. Nikolaev]
{Igor V. Nikolaev$^1$}

\address{$^{1}$ Department of Mathematics and Computer Science, St.~John's University, 8000 Utopia Parkway,  
New York,  NY 11439, United States.}
\email{\textcolor[rgb]{0.00,0.00,0.84}{igor.v.nikolaev@gmail.com}}

\dedicatory{All data are available as part of the manuscript}




\begin{abstract}
We recast the local factors of the Hasse-Weil zeta function at infinity
in terms of the Cuntz-Pimsner algebras.  
The nature of such factors is an open problem studied by Deninger and Serre.

\end{abstract}

\maketitle

\section{Introduction}
Let $V$ be an $n$-dimensional smooth projective variety over a number field $k$
and let $V(\mathbf{F}_q)$ be a good reduction of $V$ modulo the prime ideal 
corresponding to  $q=p^r$.  Recall that the local zeta  $Z_q(u):=\exp\left(\sum_{m=1}^{\infty}
|V(\mathbf{F}_q)| \frac{u^m}{m}\right)$ is a rational function
\begin{equation}\label{eq1.1}
Z_q(u)=\frac{P_1(u)\dots P_{2n-1}(u)}{P_0(u)\dots P_{2n}(u)},
\end{equation}
where $P_0(u)=1-u$ and $P_{2n}(u)=1-q^nu$.
Each  $P_i(u)$ is   the characteristic polynomial of the Frobenius 
endomorphism $Fr_q^i : ~(a_1,\dots, a_n)\mapsto (a_1^q,\dots, a_n^q)$ 
acting on  the $i$-th $\ell$-adic cohomology group  $H^i(V)$ of variety $V$.  
 The number of  points on $V(\mathbf{F}_q)$  is given by the Lefschetz trace formula 
$|V(\mathbf{F}_q)|=\sum_{i=0}^{2n}(-1)^i  ~tr~(Fr^i_q)$, 
where  $tr$ is the trace of  endomorphism  $Fr^i_q$ [Hartshorne 1977]  \cite[pp. 454-457]{H}. 
The Hasse-Weil zeta function of $V$ is an infinite product
\begin{equation}\label{eq1.2}
Z_V(s)=\prod_p Z_q(p^{-s}),  \quad s\in\mathbf{C}, 
\end{equation}
where $p$ runs through all but a finite set of primes. 
Such a function encodes arithmetic of the variety $V$.
For example, if $E$ is an elliptic curve over $\mathbf{Q}$
then $Z_E(s)=\frac{\zeta(s)\zeta(s-1)}{L(E,s)}$,
where the order of zero of function $L(E,s)$ at $s=1$
is conjectured to be equal the rank of $E$.

Recall that a fundamental  analogy between number fields and function 
fields predicts a prime $p=\infty$ in formula (\ref{eq1.1}). 
It was a mystery  how the factor $Z_{\infty}(u)$ looks  like. 
The problem was  studied by  Serre who constructed local factors  $\Gamma_V^i(s)$ realizing  the analogy.
The goal was achieved in terms of  the  $\Gamma$-functions attached to the Hodge structure on $V$
 [Serre 1970] \cite{Ser1}.   To define $\Gamma_V^i(s)$  in a way similar to  finite primes,
Deninger introduced  an infinite-dimensional cohomology $H^i_{ar}(V)$
and  an action of Frobenius endomorphism $Fr_{\infty}^i:   H^i_{ar}(V)\to H^i_{ar}(V)$, 
such that  $\Gamma_V^i(s)\equiv \left(char~Fr_{\infty}^i\right)^{-1}(s)$, where $char ~Fr_{\infty}^i$ is   the characteristic polynomial of $Fr_{\infty}^i$
[Deninger 1991]  \cite[Theorem 4.1]{Den1}.

The aim of our note is to recast  $\Gamma_V^i(s)$
 in terms of the Cuntz-Pimsner algebras [Pask \& Raeburn 1996] \cite{PasRae1}. 
Namely, let $\mathscr{A}_V$ be the Serre $C^*$-algebra of
$V$ \cite[Section 5.3.1]{N}. 
Recall \cite[Lemma 4]{Nik1} that $tr~(Fr^i_q)=~tr~(Mk^i_q)$,
where $Mk^i_q$ is the Markov endomorphism of a lattice 
$\Lambda_i\subseteq \tau_*(K_0(\mathscr{A}_V\otimes\mathcal{K}))\subset
\mathbf{R}$; the latter is  defined by the canonical trace $\tau$ on the $K_0$-group of  (stabilized) 
$C^*$-algebra $\mathscr{A}_V$, where $\mathcal{K}$ is the $C^*$-algebra of compact operators \cite[p. 271]{Nik1}. 
Therefore $|V(\mathbf{F}_q)|=\sum_{i=0}^{2n}(-1)^i  ~tr~(Mk^i_q)$ \cite[Theorem 1]{Nik1}
and the local zeta $Z_q(u)$ is a function of the endomorphisms $Mk^i_q$,
where $0\le i\le 2n$. 
On the other hand,   $Mk^i_q\in GL_{b_i}(\mathbf{Z})$ is given by a positive 
matrix, where $b_i$ is the $i$-th Betti number of $V$ \cite[p. 274]{Nik1}. 
We shall denote by $\mathcal{O}_{Mk^i_q}$ the  Cuntz-Krieger algebra
given by matrix  $Mk^i_q$ [Cuntz \& Krieger 1980] \cite{CunKrie1}. 
Thus the local factors $\Gamma_V^i(s)$ must correspond to 
 the  Cuntz-Krieger algebras given by the countably infinite  matrices 
 $A^i_{\infty}\in GL_{\infty}(\mathbf{Z})$.  The  $\mathcal{O}_{A_{\infty}^i}$ are called
the Cuntz-Pimsner algebras [Pask \& Raeburn 1996] \cite{PasRae1}.

Each matrix $A^i_{\infty}$ is constructed as follows.
Let $Mod~(V)$ be the moduli variety of $V$. Recall that 
an analog of $\mathscr{A}_V$ for $Mod~(V)$ is given 
by a cluster $C^*$-algebra $\mathbb{A}$, such that
$Prim~(\mathbb{A})\cong  Mod~(V)$, where $Prim~(\mathbb{A})$
is the set of two-sided primitive closed ideals of $\mathbb{A}$ endowed 
with the Jacobson topology. Moreover,  $\mathscr{A}_V\subset \mathbb{A}/I$
and  $K_0(\mathscr{A}_V)\cong K_0(\mathbb{A}/I)$,
where $I\in Prim~(\mathbb{A})$ \cite[Theorem 2]{Nik2}
\footnote{Note that the construction is given for $n=1$
\cite{Nik2} but true for the dimensions $n\ge 1$ either in the 
framework of the higher Teichm\"uller theory [Fock \& Goncharov 2006] \cite[Introduction]{FocGon1} or using the cluster varieties
[Casals \textit{et al.} 2025] \cite[Section 1.1]{CasAlt1}.}
. In other words, one gets a short exact sequence 
of the abelian groups:
\begin{equation}\label{eq1.3}
K_0(I) \buildrel i\over\hookrightarrow K_0(\mathbb{A})\buildrel p\over\to K_0(\mathscr{A}_V),
\end{equation}
where $K_0(I)\cong  K_0(\mathbb{A})\cong \mathbf{Z}^{\infty}$.
Since  $K_0(\mathscr{A}_V)\cong  K_0(\mathscr{A}_V\otimes\mathcal{K})$,
the  $\mathbb{Z}$-modules 
$\Lambda_i\subseteq \tau_*(K_0(\mathscr{A}_V\otimes\mathcal{K}))$
specified earlier,  define a pull back of (\ref{eq1.3}).  Thus one gets 
 an exact sequence of modules
$\Lambda_i^{\infty} \buildrel i_*\over\hookrightarrow \Lambda_i^{\infty} \buildrel p_*\over\to \Lambda_i$. 
Here the rank of cluster algebra $\tau^{-1}(\Lambda_{\infty}^i)$ is equal to the Betti number $b_i$ 
and  $i_*$ is the injective homomorphism  given by a matrix $A_{\infty}^i\in GL_{\infty}(\mathbf{Z})$
for each $0\le i\le 2n$.  Our main result  can be formulated as follows. 
\begin{theorem}\label{thm1.1}
For every smooth  $n$-dimensional 
projective variety $V$ over a number field $k$
there exist  the Cuntz-Pimsner algebras $\mathcal{O}_{A^i_{\infty}}$,  
such that the Hasse-Weil zeta function of $V$  
is given by the formula:
\begin{equation}\label{eq1.4}
Z_V(s)=\prod_{i=0}^{2n} \left( char ~A_{\infty}^i \right)^{(-1)^{i+1}}.
\end{equation}
\end{theorem}
The paper is organized as follows.  A brief review of the preliminary facts is 
given in Section 2. Theorem \ref{thm1.1}
proved in Section 3. An application of theorem \ref{thm1.1} is considered
in Section 4.

\section{Preliminaries}
We briefly review Deninger cohomology,  Cuntz-Pimsner algebras  and cluster  
$C^*$-algebras. 
We refer the reader to   [Deninger 1991] \cite{Den1},   \cite{Nik2} and [Pask \& Raeburn 1996] \cite{PasRae1}
 for a detailed exposition.

\subsection{Deninger cohomology}
The Hodge-Tate module is a $p$-adic generalization of the Hodge structure. 
Namely, let $G$ be the absolute Galois group of a $p$-adic field $\mathbf{Q}_p$
acting  by continuity on the algebraic completion $C$ of $\mathbf{Q}_p$.  
If $\chi$ is a cyclotomic character of $G$, then a module
generated by the integer powers of $\chi$ is called 
Hodge-Tate, see  [Fontaine 1982]\cite[Section 1.1]{Fon1} for the
details.  
Let $T:=\left(\varprojlim \mu_{p^n}\right)\otimes\mathbf{Q}_p$,
where $\mu_m$ is the $m$-th root of unity. 
The Hodge-Tate ring is defined as $B_{HT}:=C[T^{\pm 1}]$,
where $G$ acts on $T^i$ by $\chi^i$.  The Hodge filtration on the ring 
$B_{HT}$ is given by the formula $T^i C[T^{\pm 1}]$. 
Using the multi-prime numbers $(p_1,\dots,p_n)$,  one
can extend $B_{HT}$ to the multivaraible Laurent polynomials 
$C[T^{\pm 1}]$, where $T=(x_1,\dots, x_n)$.

Deninger's idea is to replace the ring $B_{HT}$ over $C$ 
by a ring $B_{ar}$ of the Laurent polynomials over the
archimedian place $\mathbf{R}$ 
 [Deninger 1991] \cite[Section 3]{Den1}.  
Deninger cohomology of a smooth projective variety $V$ is defined by the formula
\begin{equation}\label{eq2.1}
H^i_{ar}(V)=\mathbb{D}(B^i_{ar}),
\end{equation}
where $B^i_{ar}$ is the $i$-th cohomology of $V$ viewed as a real 
Hodge structure and $\mathbb{D}$ is a functor from the category 
of Hodge structures to an additive category of modules defined by the 
derivation $\Theta=T\frac{d}{dT}$ on the ring  $B_{ar}$. 
The following fundamental result  relates  the Deninger cohomology 
and the Serre local factors $\Gamma_V^i(s)$. 
\begin{theorem}\label{thm2.1}
{\bf (\cite[Theorem 4.1]{Den1})}
The derivation $\Theta$ induces an endomorphism $Fr_{\infty}^i:
H^i_{ar}(V)\to H^i_{ar}(V)$, such that 
\begin{equation}\label{eq2.2}
\left(char~Fr_{\infty}^i\right)^{-1}(s)\equiv \Gamma_V^i(s). 
\end{equation}
\end{theorem}
\begin{remark}\label{rmk2.2}
In what follows,  all  determinants  are the regularized determinants 
of the countably infinite-dimensional matrices in the sense of  [Deninger 1991]  \cite[Section 1]{Den1}. 
Thus the polynomial $char ~Fr_{\infty}^i := \det ~(Fr_{\infty}^i -sI)$ in (\ref{eq2.2}) 
is well defined. 
\end{remark}

\subsection{K-theory of $C^*$-algebras}
The $C^*$-algebra is an algebra  $\mathscr{A}$ over $\mathbf{C}$ with a norm 
$a\mapsto ||a||$ and an involution $\{a\mapsto a^* ~|~ a\in \mathscr{A}\}$  such that $\mathscr{A}$ is
complete with  respect to the norm, and such that $||ab||\le ||a||~||b||$ and $||a^*a||=||a||^2$ for every  $a,b\in \mathscr{A}$.  
Each commutative $C^*$-algebra is  isomorphic
to the algebra $C_0(X)$ of continuous complex-valued
functions on some locally compact Hausdorff space $X$. 
Any other  algebra $\mathscr{A}$ can be thought of as  a noncommutative  
topological space.

By $M_{\infty}(\mathscr{A})$ 
one understands the algebraic direct limit of the $C^*$-algebras 
$M_n(\mathscr{A})$ under the embeddings $a\mapsto ~\mathbf{diag} (a,0)$. 
The direct limit $M_{\infty}(\mathscr{A})$  can be thought of as the $C^*$-algebra 
of infinite-dimensional matrices whose entries are all zero except for a finite number of the
non-zero entries taken from the $C^*$-algebra $\mathscr{A}$.
Two projections $p,q\in M_{\infty}(\mathscr{A})$ are equivalent, if there exists 
an element $v\in M_{\infty}(\mathscr{A})$,  such that $p=v^*v$ and $q=vv^*$. 
The equivalence class of projection $p$ is denoted by $[p]$.   
We write $V(\mathscr{A})$ to denote all equivalence classes of 
projections in the $C^*$-algebra $M_{\infty}(\mathscr{A})$, i.e.
$V(\mathscr{A}):=\{[p] ~:~ p=p^*=p^2\in M_{\infty}(\mathscr{A})\}$. 
The set $V(\mathscr{A})$ has the natural structure of an abelian 
semi-group with the addition operation defined by the formula 
$[p]+[q]:=\mathbf{diag}(p,q)=[p'\oplus q']$, where $p'\sim p, ~q'\sim q$ 
and $p'\perp q'$.  The identity of the semi-group $V(\mathscr{A})$ 
is given by $[0]$, where $0$ is the zero projection. 
By the $K_0$-group $K_0(\mathscr{A})$ of the unital $C^*$-algebra $\mathscr{A}$
one understands the Grothendieck group of the abelian semi-group
$V(\mathscr{A})$, i.e. a completion of $V(\mathscr{A})$ by the formal elements
$[p]-[q]$.  The image of $V(\mathscr{A})$ in  $K_0(\mathscr{A})$ 
is a positive cone $K_0^+(\mathscr{A})$ defining  the order structure $\le$  on the  
abelian group  $K_0(\mathscr{A})$. The pair   $\left(K_0(\mathscr{A}),  K_0^+(\mathscr{A})\right)$
is known as a dimension group of the $C^*$-algebra $\mathscr{A}$  [Blackadar 1986] \cite[Chapter III]{B}.

\subsection{AF-algebras}
An Approximately Finite-dimensional $C^*$-algebra (AF-algebra) is defined to
be the  norm closure of an ascending sequence of   finite dimensional
$C^*$-algebras $M_n$,  where  $M_n$ is the $C^*$-algebra of the $n\times n$ matrices
with entries in $\mathbf{C}$. Here the index $n=(n_1,\dots,n_k)$ represents
the  semi-simple matrix algebra  $M_i=M_{i_1}\oplus\dots\oplus M_{i_k}$.
The ascending sequence mentioned above  can be written as 
\displaymath
M_{i_1}\buildrel\rm\varphi_1\over\longrightarrow M_{i_2}
   \buildrel\rm\varphi_2\over\longrightarrow\dots,
 \enddisplaymath  
where $M_{i_k}$ are the finite dimensional $C^*$-algebras and
$\varphi_i$ the homomorphisms between such algebras.  
If $\varphi_i=Const$, then the AF-algebra $\mathscr{A}$ is called 
{\it stationary}. 
The homomorphisms $\varphi_i$ can be arranged into  a graph as follows. 
Let  $M_i=M_{i_1}\oplus\dots\oplus M_{i_k}$ and 
$M_{i'}=M_{i_1'}\oplus\dots\oplus M_{i_k'}$ be 
the semi-simple $C^*$-algebras and $\varphi_i: M_i\to M_{i'}$ the  homomorphism. 
One has  two sets of vertices $V_{i_1},\dots, V_{i_k}$ and $V_{i_1'},\dots, V_{i_k'}$
joined by  $a_{rs}$ edges  whenever the summand $M_{i_r}$ contains $a_{rs}$
copies of the summand $M_{i_s'}$ under the embedding $\varphi_i$. 
As $i$ varies, one obtains an infinite graph called the   Bratteli diagram of the
AF-algebra.  The matrix $A=(a_{rs})$ is known as  a  partial multiplicity matrix;
an infinite sequence of $A_i$ defines a unique AF-algebra.
If   $\mathbb{A}$ is a stationary AF-algebra, then   $A_i=Const$
for all $i\ge 1$.  
The  dimension group $\left(K_0(\mathbb{A}),  K_0^+(\mathbb{A})\right)$  is a complete invariant of the
Morita equivalence class of the AF-algebra $\mathbb{A}$ [Blackadar 1986] \cite[Section 7.3]{B}.

\subsection{Cuntz-Pimsner algebras}
Recall that 
the Cuntz-Krieger algebra $\mathcal{O}_A$ is a  $C^*$-algebra
generated by the  partial isometries $s_1,\dots, s_n$ which satisfy  the relations
\begin{equation}
\left\{
\begin{array}{ccc}
s_1^*s_1 &=& a_{11} s_1s_1^*+a_{12} s_2s_2^*+\dots+a_{1n}s_ns_n^*\\ 
s_2^*s_2 &=& a_{21} s_1s_1^*+a_{22} s_2s_2^*+\dots+a_{2n}s_ns_n^*\\ 
                  &\dots&\\
s_n^*s_n &=& a_{n1} s_1s_1^*+a_{n2} s_2s_2^*+\dots+a_{nn}s_ns_n^*,             
\end{array}
\right.
\end{equation}
where $A=(a_{ij})$ is a square matrix with  $a_{ij}\in \{0, 1, 2, \dots \}$. 
(Note that the original definition of $\mathcal{O}_A$ says that  $a_{ij}\in \{0, 1\}$
but is known to be extendable to all non-negative integers [Cuntz \& Krieger 1980] \cite{CunKrie1}.)   
Such algebras  appear naturally in the study of local factors \cite{Nik1}.

The Cuntz-Pimsner algebra is a generalization
of  $\mathcal{O}_A$ to the countably infinite matrices $A_{\infty}\in GL_{\infty}(\mathbf{Z})$  
[Pask \& Raeburn 1996] \cite{PasRae1}.
Recall that the matrix $A_{\infty}$ is called row-finite,  if for each $i\in\mathbf{N}$
the number of $j\in\mathbf{N}$ with $a_{ij}\ne 0$ is finite.  The matrix $A$ is 
said to be irreducible, if some power of $A$ is a strictly positive matrix and $A$ is not a
permutation matrix.  It is known that if  $A_{\infty}$ is row-finite and irreducible, then the 
Cuntz-Pimsner algebra 
 $\mathcal{O}_{A_{\infty}}$ is a well-defined  and simple 
[Pask \& Raeburn 1996] \cite[Theorem 1]{PasRae1}.
An AF-core $\mathscr{F}\subset \mathcal{O}_{A_{\infty}}$ is an Approximately Finite (AF-) $C^*$-algebra
defined by the closure of of the infinite union $\cup_{k,j} \cup_{i\in V_k^j} \mathscr{F}_k^j(i)$,
where  $\mathscr{F}_k^j(i)$ are finite-dimensional $C^*$-algebras 
built from matrix $A_{\infty}$, see [Pask \& Raeburn 1996] \cite[Definition 2.2.1]{PasRae1}
for the details. We shall denote 
by $\alpha: \mathcal{O}_{A_{\infty}}\to \mathcal{O}_{A_{\infty}}$ an automorphism 
acting on the generators $s_i$ of $\mathcal{O}_{A_{\infty}}$ by
to the formula $\alpha_z(s_i)=zs_i$, where $z$ is a complex number of the absolute value $|z|=1$.
Thus one gets an action of the abelian group $\mathbb{T}\cong\mathbf{R}/\mathbf{Z}$ on  $\mathcal{O}_{A_{\infty}}$. 
It follows from the Takai duality [Pask \& Raeburn 1996] \cite[p. 432]{PasRae1} that:
\begin{equation}\label{eq2.4} 
\mathscr{F}\rtimes_{\hat\alpha}\mathbb{T}\cong \mathcal{O}_{A_{\infty}}\otimes\mathcal{K},
\end{equation}
where $\hat\alpha$ is the Takai dual of $\alpha$ and $\mathcal{K}$ is the
$C^*$-algebra of compact operators.  Using (\ref{eq2.4}) one can calculate the the 
$K$-theory of  $\mathcal{O}_{A_{\infty}}$.
\begin{theorem}\label{thm2.2}
{\bf (\cite[Theorem 3]{PasRae1})}
If $A_{\infty}$ is row-finite irreducible matrix, then there exists an exact 
sequence of the abelian groups:
\begin{equation}\label{eq2.5} 
0\to K_1(\mathcal{O}_{A_{\infty}})\to \mathbf{Z}^{\infty}\buildrel 1-A_{\infty}^t\over\longrightarrow 
 \mathbf{Z}^{\infty}\buildrel i_*\over\longrightarrow  K_0(\mathcal{O}_{A_{\infty}})\to 0, 
\end{equation}
so that $K_0(\mathcal{O}_{A_{\infty}})\cong  \mathbf{Z}^{\infty}/(1-A_{\infty}^t) \mathbf{Z}^{\infty}$ and 
 $K_1(\mathcal{O}_{A_{\infty}})\cong Ker ~(1-A_{\infty}^t)$, where  $A_{\infty}^t$ is the transpose 
 of  $A_{\infty}$ and $i: \mathscr{F}\hookrightarrow \mathcal{O}_{A_{\infty}}$.
 Moreover, the Grothendieck semigroup $K_0^+(\mathscr{F})\cong \varinjlim (\mathbf{Z}^{\infty}, A_{\infty}^t)$. 
 \end{theorem}

\subsection{Cluster $C^*$-algebras}
The cluster algebra  of rank $n$ 
is a subring  $\mathcal{A}(\mathbf{x}, B)$  of the field  of  rational functions in $n$ variables
depending  on  variables  $\mathbf{x}=(x_1,\dots, x_n)$
and a skew-symmetric matrix  $B=(b_{ij})\in M_n(\mathbf{Z})$.
The pair  $(\mathbf{x}, B)$ is called a  seed.
A new cluster $\mathbf{x}'=(x_1,\dots,x_k',\dots,  x_n)$ and a new
skew-symmetric matrix $B'=(b_{ij}')$ is obtained from 
$(\mathbf{x}, B)$ by the   exchange relations [Williams 2014]  \cite[Definition 2.22]{Wil1}:
\begin{eqnarray}\label{eq2.3}
x_kx_k'  &=& \prod_{i=1}^n  x_i^{\max(b_{ik}, 0)} + \prod_{i=1}^n  x_i^{\max(-b_{ik}, 0)},\cr 
b_{ij}' &=& 
\begin{cases}
-b_{ij}  & \mbox{if}   ~i=k  ~\mbox{or}  ~j=k\cr
b_{ij}+\frac{|b_{ik}|b_{kj}+b_{ik}|b_{kj}|}{2}  & \mbox{otherwise.}
\end{cases}
\end{eqnarray}
The seed $(\mathbf{x}', B')$ is said to be a  mutation of $(\mathbf{x}, B)$ in direction $k$.
where $1\le k\le n$.  The  algebra  $\mathcal{A}(\mathbf{x}, B)$ is  generated by the 
cluster  variables $\{x_i\}_{i=1}^{\infty}$
obtained from the initial seed $(\mathbf{x}, B)$ by the iteration of mutations  in all possible
directions $k$.   The  Laurent phenomenon
 says  that  $\mathcal{A}(\mathbf{x}, B)\subset \mathbf{Z}[\mathbf{x}^{\pm 1}]$,
where  $\mathbf{Z}[\mathbf{x}^{\pm 1}]$ is the ring of  the Laurent polynomials in  variables $\mathbf{x}=(x_1,\dots,x_n)$
 [Williams 2014]  \cite[Theorem 2.27]{Wil1}.
In particular, each  generator $x_i$  of  the algebra $\mathcal{A}(\mathbf{x}, B)$  can be 
written as a  Laurent polynomial in $n$ variables with the   integer coefficients.

 The cluster algebra  $\mathcal{A}(\mathbf{x}, B)$  has the structure of an additive abelian
semigroup (a.k.a. the \textit{Grothendieck semigroup})  consisting of the Laurent polynomials with positive coefficients. 
In other words,  the $\mathcal{A}(\mathbf{x}, B)$ is a dimension group, see [Blackadar 1986] \cite[Section 7.3]{B}  or  
\cite[Definition 3.5.2]{N}.
The cluster $C^*$-algebra  $\mathbb{A}(\mathbf{x}, B)$  is   an  AF-algebra,  
such that $K_0(\mathbb{A}(\mathbf{x}, B))\cong  \mathcal{A}(\mathbf{x}, B)$.

\subsubsection{Cluster $C^*$-algebra $\mathbb{A}(S_{g,n})$}
Denote by $S_{g,n}$  the Riemann surface   of genus $g\ge 0$  with  $n\ge 0$ cusps.
 Let   $\mathcal{A}(\mathbf{x},  S_{g,n})$ be the cluster algebra 
 coming from  a triangulation of the surface $S_{g,n}$   [Williams 2014]  \cite[Section 3.3]{Wil1}. 
 We shall denote by  $\mathbb{A}(S_{g,n})$  the corresponding cluster $C^*$-algebra. 
 Let $T_{g,n}$ be the Teichm\"uller space of the surface $S_{g,n}$,
i.e. the set of all complex structures on $S_{g,n}$ endowed with the 
natural topology. The geodesic flow $T^t: T_{g,n}\to T_{g,n}$
is a one-parameter  group of matrices $\mathbf{diag}  (e^t, e^{-t})$
acting on the holomorphic quadratic differentials on the Riemann surface $S_{g,n}$. 
Such a flow gives rise to a one parameter group of automorphisms 
$\sigma_t: \mathbb{A}(S_{g,n})\to \mathbb{A}(S_{g,n})$
called the Tomita-Takesaki flow on the AF-algebra $\mathbb{A}(S_{g,n})$. 
Denote by $Prim~\mathbb{A}(S_{g,n})$ the space of all primitive ideals 
of $\mathbb{A}(S_{g,n})$ endowed with the Jacobson topology. 
Recall (\cite{Nik2}) that each primitive ideal has a parametrization by a vector 
$\Theta\in \mathbf{R}^{6g-7+2n}$ and we write it 
$I_{\Theta}\in Prim~\mathbb{A}(S_{g,n})$
\begin{theorem}\label{thm2.3}
{\bf (\cite{Nik2})}
There exists a homeomorphism
$h:  Prim~\mathbb{A}(S_{g,n})\times \mathbf{R}\to \{U\subseteq  T_{g,n} ~|~U~\hbox{{\sf is generic}}\}$
given by the formula $\sigma_t(I_{\Theta})\mapsto S_{g,n}$;  the set $U=T_{g,n}$ if and only if
$g=n=1$.   The $\sigma_t(I_{\Theta})$
is an ideal of  $\mathbb{A}(S_{g,n})$ for all $t\in \mathbf{R}$ and 
 the quotient  algebra  $\mathbb{A}(S_{g,n})/\sigma_t(I_{\Theta})$
is  a non-commutative coordinate ring  of  the Riemann surface  $S_{g,n}$.  
\end{theorem}

\section{Proof of theorem \ref{thm1.1}}
For the sake of clarity, let us outline the main ideas. 
Let $Fr^i_{\infty}$ be the Frobenius endomorphism of the $i$-th  Deninger cohomology group $H^i_{ar}(V)$ 
as stated in Theorem \ref{thm2.1}. 
Recall  (\ref{eq2.1}) that $H^i_{ar}(V)$
is  the additive group of the ring of Laurent polynomilas $\mathbf{R}[\mathbf{x}_i^{\pm 1}]$,
where $\mathbf{x}_i=(x_1,\dots,x_{b_i})$.  
We shall use the following definitions and notations. 
\begin{definition}
For each $1\le i\le 2n$,  we let  $\mathbb{A}^i$ be a cluster $C^*$-algebra whose Grothendieck semigroup
is given by the Laurent polynomials  $\mathbf{Z}[\mathbf{x}_i^{\pm 1}]$ with  positive coefficients. 
Likewise,  we denote by $\mathbb{A}_V^i$ an $AF$-algebra whose Grothendieck semigroup
is defined by the lattice $\Lambda_i\subset\mathbf{R}$, i.e. by the additive semigroup of positive reals in $\Lambda_i$ 
 \cite[p. 271]{Nik1}. 
\end{definition}
Consider the kernel of endomorphism  $Fr^i_{\infty}$ in  $\mathbf{Z}[\mathbf{x}_i^{\pm 1}]$.
The latter  is  a two-sided primitive ideal $I_F^i$ (called the $i$-th Fontaine ideal) in the cluster
$C^*$-algebra $\mathbb{A}^i$. We prove that the quotient $\mathbb{A}^i/I_F^i$ and the $AF$-algebra $\mathbb{A}_V^i$ are 
Morita equivalent (Lemma \ref{lm3.1}). 
Next  it is proved that matrix $A_{\infty}^i$ is conjugate to $Fr_{\infty}^i$ in 
 the group $GL_{\infty}(\mathbf{Z})$  (Lemma \ref{lm3.3}). 
 In particular, $char~ A_{\infty}^i\equiv char  ~Fr_{\infty}^i$ for all $0\le i\le 2n$ (Corollary \ref{cor3.4}). 
The rest of the proof follows from Theorem \ref{thm2.1}, see Lemma \ref{lm3.5}. 
Let us pass to a detailed argument.

\begin{lemma}\label{lm3.1}
 The quotient $\mathbb{A}^i/I_F^i$ is stably isomorphic (Morita equivalent) to 
 the $AF$-algebra $\mathbb{A}_V^i$.  
 \end{lemma}
\begin{proof}
(i) Let us show that if projective varieties $V, V'$ are isomorphic 
over the number field $k$, then there exists a ring automorphism 
$\phi$ of $\mathbb{A}^i$ such that the corresponding Fontaine ideal
$I_{F'}^i=\phi(I_F^i)$. Indeed, let $V\to V'$ be an isomorphism between
projective varieties $V$ and $V'$.  The cohomology functor induces 
an isomorphism $\phi : H_{ar}^i(V)\to  H_{ar}^i(V')$ of the corresponding 
Deninger cohomology groups.  Recall that $H_{ar}^i(V)\cong \mathbf{R}[\mathbf{x}^{\pm 1}]$
and since the isomorphism of $V$ is defined over a number field $k$,  one 
gets an isomorphism $\phi:\mathbf{Z}[\mathbf{x}^{\pm 1}]\to \mathbf{Z}[\mathbf{y}^{\pm 1}]$.
(Note that  group isomorphism $\phi$  extends to a ring  isomorphism by choice of a monomial basis in the ring
of the Laurent polynomials, and vice versa.) 
Recall that $K_0(\mathbb{A}^i)\cong \mathbf{Z}[\mathbf{x}^{\pm 1}]$ and $K$-theory is a functor;
thus one gets an an automorphism $\phi:   \mathbb{A}^i\to \mathbb{A}^i$. 
It remains to notice that the endomorphism $Fr^i_{\infty}: \mathbf{Z}[\mathbf{x}^{\pm 1}]
\to \mathbf{Z}[\mathbf{x}^{\pm 1}]$ commutes with $\phi$ and therefore  
$\phi(Fr^i_{\infty}(\mathbf{Z}[\mathbf{x}^{\pm 1}]))=Fr_{\infty}^i(\mathbf{Z}[\mathbf{y}^{\pm 1}])$.   
By definition $K_0(I_F^i) \cong Fr^i_{\infty}(\mathbf{Z}[\mathbf{x}^{\pm 1}])$ and thus 
the Fontaine ideal $I_{F'}^i=\phi(I_F^i)$.

\bigskip
(ii) Let $I^i_F\subset \mathbb{A}^i$ be an $i$-th Fontaine ideal. Since $I^i_F$
is a primitive two-sided ideal, the quotient $C^*$-algebra   $\mathbb{A}^i/I_F^i$ 
is simple. It follows from item (i) that isomorphisms of $V$ over $k$ correspond 
to the $C^*$-isomorphisms of the algebra $\mathbb{A}^i/I_F^i$. 

\bigskip
(iii) On the other hand,  we have a lattice 
$\Lambda_i\subseteq \tau_*(K_0(\mathscr{A}_V\otimes\mathcal{K}))\subset
\mathbf{R}$ , where the rank of $\Lambda_i$ is equal to the $i$-th Betti number $b_i$
of $V$ \cite[p.271]{Nik1}.  It is well known that if projective varieties $V, V'$ are isomorphic 
over the number field $k$, then their Serre $C^*$-algebras $\mathscr{A}_V, \mathscr{A}_{V'}$
must be stably isomorphic (even isomorphic) \cite[Section 5.3.1]{N}. 
Since the $K_0$-groups are invariant under  the stable isomorphisms, the lattices  $\tau_*(K_0(\mathscr{A}_V\otimes\mathcal{K}))
\equiv\tau_*(K_0(\mathscr{A}_{V'}\otimes\mathcal{K}))$ and $\Lambda_i\equiv \Lambda_i'$ as subsets of the real line.
By definition $K_0(\mathbb{A}_V^i)=\Lambda_i$, so that the $AF$-algebras  $\mathbb{A}_V^i$ and 
  $\mathbb{A}_{V'}^i$ are isomorphic.

  \bigskip
  (iv) To finish the proof, it remains to compare the results of items (ii) and (iii). 
  Indeed, we constructed two covariant functors $V\mapsto \mathbb{A}^i/I_F^i$ 
  and $V\mapsto \mathbb{A}_V^i$ from smooth  $n$-dimensional projective
  varieties $V$ to the category of $AF$-algebras. But all morphisms
  in the latter category are stable isomorphisms between the $AF$-algebras,
  i.e. $\left(\mathbb{A}^i/I_F^i\right)\otimes\mathcal{K}\cong 
  \mathbb{A}_V^i\otimes\mathcal{K}$.  
  
  \bigskip
  Lemma \ref{lm3.1} is
  proved.
  \end{proof}

\begin{corollary}\label{cor3.2}
Cluster algebra $K_0(\mathbb{A}^i)$ has rank equal to the $i$-th Betti number of $V$.
\end{corollary}
\begin{proof}
It is known that the rank of lattice $\Lambda_i$ is equal to the $i$-th Betti number $b_i$
of variety $V$  \cite[p.271]{Nik1}.   Since  $K_0(\mathbb{A}_V^i)\cong \Lambda_i$
and $\left(\mathbb{A}^i/I_F^i\right)\otimes\mathcal{K}\cong 
  \mathbb{A}_V^i\otimes\mathcal{K}$, 
  we conclude that $K_0(\mathbb{A}^i/I_F^i)\cong \Lambda_i$
  and thus the rank of $K_0(\mathbb{A}^i)$ is equal to $b_i$. 
  \end{proof}

\begin{lemma}\label{lm3.3}
There exists a simple Cuntz-Pimsner algebra  $\mathcal{O}_{A_{\infty}^i}$ such that:
\begin{equation}
 \mathcal{O}_{A_{\infty}^i}\otimes\mathcal{K}\cong I_F^i\rtimes_{\hat\alpha^i}\mathbb{T},
\end{equation}
where $A_{\infty}^i\in GL_{\infty}(\mathbf{Z})$ is conjugate to the matrix $Fr_{\infty}^i$
and  $I_F^i$ is the $i$-th Fontaine ideal of $\mathbb{A}^i$.
\end{lemma}
\begin{proof}
(i) For an $i$-th Fontaine ideal $I_F^i\subset\mathbb{A}^i$,  let us calculate the semi-group
$K^+_0(I_F^i)$.  It is easy to see, that $K_0(I_F^i)\cong \mathbf{Z}^{\infty}$
and the corresponding Grothendieck semigroup $K^+_0(I_F^i)\cong\varinjlim (\mathbf{Z}^{\infty}, Fr_{\infty}^i)$,
where the injective limit is taken by the iterations of the endomorphism $Fr_{\infty}^i$ acting on $\mathbf{Z}^{\infty}$.

\bigskip
(ii) On the other hand, if  $\mathscr{F}^i\subset  \mathcal{O}_{A_{\infty}^i}$
is the core $AF$-algebra of a Cuntz-Pimsner algebra  $\mathcal{O}_{A_{\infty}^i}$,
then $K_0^+(\mathscr{F}^i)\cong \varinjlim (\mathbf{Z}^{\infty}, (A_{\infty}^i)^t)$,
see Theorem \ref{thm2.2}. 

\bigskip
(iii) We  now define matrix $A_{\infty}^i\in GL_{\infty}(\mathbf{Z})$  so that: 
\begin{equation}\label{eq3.2}
K_0^+(\mathscr{F}^i)\cong K^+_0(I_F^i),
\end{equation}
where $\cong$ is an isomorphism of the Grothendieck semigroups, i.e. 
an order-isomorphism of the corresponding positive cones.   

\bigskip
(iv) It follows from (\ref{eq3.2}) that 
$I_F^i\rtimes_{\hat\alpha^i}\mathbb{T}\cong \mathcal{O}_{A_{\infty}^i}\otimes\mathcal{K}$,
see formula (\ref{eq2.4}).  Moreover, an isomorphism  $\varinjlim (\mathbf{Z}^{\infty}, (A_{\infty}^i)^t)\cong 
\varinjlim (\mathbf{Z}^{\infty}, Fr_{\infty}^i)$ implies that  matrices 
$A_{\infty}^i$ and $Fr_{\infty}^i$ are conjugate in $GL_{\infty}(\mathbf{Z})$. 

\bigskip
(v)  Since the determinant $\det ~(Fr_{\infty}^i -sI)$ is regular (Remark \ref{rmk2.2}),
we conclude that the conjugate matrix $A_{\infty}^i$ must be row-finite and irreducible,
i.e.  $\mathcal{O}_{A_{\infty}^i}$ is a correctly defined simple Cuntz-Pimsner algebra.

\bigskip
Lemma \ref{lm3.3} is proved.
\end{proof}

\begin{corollary}\label{cor3.4}
$char~ A_{\infty}^i\equiv char  ~Fr_{\infty}^i$. 
\end{corollary}
\begin{proof}
The characteristic polynomial $char~ A_{\infty}^i=\det (A_{\infty}^i-sI)$ is invariant 
of the conjugacy class of matrix $A_{\infty}^i$. We conclude from 
Lemma \ref{lm3.3} that $char~ A_{\infty}^i\equiv char  ~Fr_{\infty}^i$.
(The converse is false in general.) 
Corollary \ref{cor3.4} is proved.

\end{proof}

\begin{lemma}\label{lm3.5}
$Z_V(s)=\prod_{i=0}^{2n} \left( char ~A_{\infty}^i \right)^{(-1)^{i+1}}$. 
\end{lemma}
\begin{proof}
(i) Recall that 
\begin{equation}\label{eq3.3}
Z_V(s)=\prod_{i=0}^{2n}\left(\Gamma_V^i(s)\right)^{(-1)^i},
\end{equation}
where $\Gamma_V^i(s)$ is the $i$-th Serre local factor at $p=\infty$
 [Serre 1970] \cite[Section 3]{Ser1}.  In view of Deninger's Theorem \ref{thm2.1}
 we can substitute  $\Gamma_V^i(s)\equiv \left(char~Fr_{\infty}^i\right)^{-1}(s)$ 
 so that the Hasse-Weil zeta function (\ref{eq3.3}) becomes:
\begin{equation}\label{eq3.4}
Z_V(s)=\prod_{i=0}^{2n}\left(char~Fr_{\infty}^i\right)^{(-1)^{i+1}}.
\end{equation}

\bigskip
(ii) On the other hand, there exist  Cuntz-Pimsner algebras  
$\mathcal{O}_{A_{\infty}^i}$, such that 
$char~ A_{\infty}^i\equiv char  ~Fr_{\infty}^i$ (Corollary \ref{cor3.4}). 
Thus one can write the Hasse-Weil zeta function (\ref{eq3.4}) in the form:
\begin{equation}\label{eq3.5}
Z_V(s)=\prod_{i=0}^{2n} \left( char ~A_{\infty}^i \right)^{(-1)^{i+1}}. 
\end{equation}

\bigskip
Lemma \ref{lm3.5} is proved.
\end{proof}

\bigskip
Theorem \ref{thm1.1} follows from  Lemmas \ref{lm3.3} and  \ref{lm3.5}.

\section{Riemann zeta function}
Let us point out a relation between the Cuntz-Pimsner algebra $\mathcal{O}_{A_{\infty}}$
and non-trivial zeroes of the Riemann zeta function $\zeta (s)$.  
If  $V$ is a curve, then $n=1$  and  formula (\ref{eq1.4}) for the Hasse-Weil zeta 
function  can be written as:
\begin{equation}\label{eq4.1}
Z_V(s)=\prod_{i=0}^{2} \left( char ~A_{\infty}^i \right)^{(-1)^{i+1}}=
\frac{char ~A^1_{\infty}}{char ~A^0_{\infty}  ~char ~A^2_{\infty}}.
\end{equation}
Moreover, $char ~A^0_{\infty}=\frac{s}{2\pi}$ and  $char ~A^2_{\infty}=\frac{s-1}{2\pi}$ 
[Deninger 1992]  \cite[Section 3]{Den2}.  Thus one can write (\ref{eq4.1}) in the form:
\begin{equation}\label{eq4.2}
(2\pi)^{-2} s(s-1)Z_V(s)=char ~A^1_{\infty}.
\end{equation}

\bigskip
On the other hand, the Hasse-Weil zeta function can be linked 
to the Riemann zeta function $\zeta (s)$ by the well known formula:
\begin{equation}\label{eq4.3}
Z_V(s)=2^{-\frac{1}{2}}\pi^{-\frac{s}{2}}\Gamma\left(\frac{s}{2}\right)\zeta (s),
\end{equation}
where $\Gamma\left(\frac{s}{2}\right)$ is the gamma function.
We can use (\ref{eq4.3}) to exclude $Z_V(s)$ from  (\ref{eq4.2}):
\begin{equation}\label{eq4.4}
2^{-\frac{5}{2}}\pi^{\frac{-s-4}{2}}\Gamma\left(\frac{s}{2}\right)s(s-1)\zeta (s)=char ~A^1_{\infty}.
\end{equation}

\bigskip
It follows from (\ref{eq4.4}) that non-trivial zeros of the Riemann zeta function 
coincide with the roots of characteristic polynomial of the matrix $A^1_{\infty}$ defining 
the Cuntz-Pimsner algebra   $\mathcal{O}_{A_{\infty}^1}$.  Thus the row-finite and irreducible
matrices are proper candidates for Hilbert's idea  to settle the   
 Riemann Hypothesis (RH) via spectra of the self-adjoint operators. 
\begin{remark}
Assuming validity of the RH, one gets an interesting conclusion that the matrix 
$A_{\infty}^1$ should be: (i) row- and column-finite and (ii) the entries $a_{ij}$ of  $A_{\infty}^1$ 
can be half- and negative integers. Indeed, if $s\in\{\frac{1}{2}+it ~|~t\in\mathbf{R}\}$,
then $char~  A_{\infty}^1(s)=\det~(A_{\infty}^1 -sI_{\infty})=\det~(A_{\infty}^1-\frac{1}{2}I_{\infty}-itI_{\infty})$,
i.e. the matrix   $A_{\infty}^1-\frac{1}{2}I_{\infty}$ must be  skew-symmetric. 
In particular, row-finiteness implies column-finiteness of matrix $A_{\infty}^1$  and  entries $a_{ij}$
 can take  half-integer  and negative integer values.  Although the Cuntz-Pimsner algebras satisfying condition (ii) 
are not of the type studied in    [Pask \& Raeburn 1996] \cite{PasRae1}, they generalize to such
preserving all essential properties of the Cuntz-Pimsner algebras [Katsura 2004] \cite{Kat1}. 
\end{remark}

\section*{Data availability}
  
  Data sharing not applicable to this article as no datasets were generated or analyzed during the current study.
   
\section*{Conflict of interest}
 On behalf of all co-authors, the corresponding author states that there is no conflict of interest.
  

\section*{Funding declaration}
The author was partly supported by the NSF-CBMS grant 2430454.

\subsection*{Acknowledgment}
 The author would like to thank the anonymous referee who provided thoughtful  comments on an earlier version of the manuscript.

\bibliographystyle{amsplain}


\end{document}